\documentclass[sort&compress,3p]{elsarticle}
\usepackage{amsmath}
\usepackage{amssymb}
\usepackage{array}
\usepackage{array}
\usepackage{graphicx}
\usepackage{tikz,ulem}
\usepackage{tikz-qtree}
\usepackage{multirow}
\usepackage{bbm}
\usepackage{dsfont}
\usepackage{enumitem}
\usepackage{subfigure}
\usepackage{hyperref}

\usepackage{algpseudocode,algorithm}
\usepackage{tikz}

\newcommand{\K}{\mathcal{K}}
\newcommand{\A}{\mathbf{A}}

\renewcommand{\Re}{\mathds{R}}
\newcommand{\I}{\mathbf{I}}
\newcommand{\V}{\mathbf{V}}
\newcommand{\J}{\mathbf{J}}

\newcommand{\one}{\mathbf{1}}
\newcommand{\Hb}{\mathbf{H}}
\newcommand{\Lb}{\mathbf{L}}
\newcommand{\Ub}{\mathbf{U}}
\newcommand{\Pb}{\mathbf{P}}
\newcommand{\h}{\mathbf{h}}
\newcommand{\Kb}{\mathbf{K}}
\newcommand{\Yb}{\mathbf{Y}}
\newcommand{\Rb}{\boldsymbol{\rho}}
\newcommand{\alphaB}{\boldsymbol{\alpha}}
\newcommand{\gammaB}{\boldsymbol{\gamma}}
\newcommand{\betaB}{\boldsymbol{\beta}}

\newtheorem{theorem}{Theorem}
\newtheorem{remark}{Remark}
\newtheorem{lemma}{Lemma}
\newtheorem{corollary}{Corollary}

\newproof{proof}{Proof}

\begin{document}
\begin{frontmatter}
  \title{Adaptive Krylov-Type Time Integration Methods}
  \author[csl,llnl]{Paul Tranquilli\corref{cor}}
  \ead{tranquilli1@llnl.gov}
  \cortext[cor]{Corresponding author}
  \address[csl]{Computational Science Laboratory, Department of Computer Science, Virginia Tech.  Blacksburg, Virginia 24060}
  \address[llnl]{Lawrence Livermore National Laboratory. Livermore, California 94550}
  \author[csl]{Ross Glandon}
  \ead{rossg42@vt.edu}  
  \author[csl]{Adrian Sandu}
  \ead{sandu@cs.vt.edu}

\begin{abstract}
The Rosenbrock-Krylov family of time integration schemes is an extension of Rosenbrock-W methods that employs a specific Krylov based approximation of the linear system solutions arising within each stage of the integrator.  This work proposes an extension of Rosenbrock-Krylov methods to address stability questions which arise for methods making use of inexact linear system solution strategies.  Two approaches for improving the stability and efficiency of Rosenbrock-Krylov methods are proposed, one through direct control of linear system residuals and the second through a novel extension of the underlying Krylov space to include stage right hand side vectors.  Rosenbrock-Krylov methods employing the new approaches show a substantial improvement in computational efficiency relative to prior implementations. 
\end{abstract}
\end{frontmatter}

\tableofcontents

\section{Introduction}
\label{Radv:sec:intro}

This work is concerned with solving initial value problems:
\begin{equation}
\label{Radv:eqn:ode}
 \frac{dy}{dt} = f(t,y)\,,~~~ t_0 \leq t \leq t_F\,, ~~~ y(t_0) = y_0\,; \quad  y(t) \in \Re^N,\, \quad f : \Re \times \Re^{N} \to \Re^N.
\end{equation}
We consider numerical discretizations of \eqref{Radv:eqn:ode} the following general form:
\begin{subequations}
 \label{Radv:eqn:genformW}
 \begin{eqnarray}
  \label{Radv:eqn:genformWstage}
   k_i & = & \varphi(h\gamma \A_n)\left(h\, F_i + h\A_n\displaystyle\sum_{j=1}^{i-1}\gamma_{i,j}k_j\right), \\
  \label{Radv:eqn:genformWrhs}
   F_i & = & f\left( y_n +  \displaystyle\sum_{j=1}^{i-1} \alpha_{i,j}k_j \right), \\
  \label{Radv:eqn:genformWynew}
   y_{n+1} & = & y_n +  \displaystyle\sum_{i=1}^s b_i k_i.
  \end{eqnarray}
\end{subequations}
where $\A_n$ is either the exact Jacobian matrix, $\J_n$, coming from the right-hand-side vector $f(t,y)$
\begin{equation}
\J_n := \left.\frac{\partial f}{\partial y}\right|_{t = t_n},
\end{equation}
or an approximation of it.  

The general form \eqref{Radv:eqn:genformW} encapsulates explicit Runge-Kutta methods when $\varphi(z) = 1$, Rosenbrock methods when $\varphi(z) = 1/(1-z)$, and exponential-Rosenbrock methods when $\varphi(z) = (e^{z}-1)/z$.
When the coefficients $\alpha_{i,j}$, $\gamma_{i,j}$, and $b_i$ are chosen such that $\A_n$ must equal the Jacobian matrix $\J_n$, the scheme \eqref{Radv:eqn:genformW} is known as either a classical Rosenbrock or exponential-Rosenbrock method depending on the choice of $\varphi(z)$.  Alternatively, the coefficients can be chosen such that the matrix $\A_n$ can be any arbitrary approximation of the Jacobian $\J_n$, and in this case the scheme \eqref{Radv:eqn:genformW} is known as either a Rosenbrock-W or exponential-Rosenbrock-W method depending once again on the choice of $\varphi(z)$.

The use of W-methods is convenient when the exact Jacobian $\J_n$ is difficult, or particularly expensive, to compute, or when iterative linear system solution methods are preferred \cite{Tranquilli2017}.  The drawback of this approach is a dramatic increase in the number of required order conditions imposed on the coefficients $\alpha_{i,j}$, $\gamma_{i,j}$, and $b_i$ in order to achieve the desired accuracy \cite{Hairer_book_II}.  For this reason, among others, we focus our attention on an alternative formulation of \eqref{Radv:eqn:genformW} called Rosenbrock-Krylov \cite{Tranquilli_2014_ROK} or exponential-Krylov \cite{Tranquilli_2014_expK} methods, in which the coefficients are chosen to permit a specific, Krylov-based, approximation $\A_n$ of the Jacobian matrix $\J_n$.  These methods have been shown to be at least efficient for a broad spectrum of problems through comparisons of relative performance against an assortment of standard integrators \cite{Tranquilli_2014_ROK, Sandu_2017_EPIRK,Tranquilli_2014_expK,Tranquilli2017,Wu_2016_ROK4E}, and have been effectively used to study engineering problems in  \cite{Bravo_2018_transcritical, wu_2018_efficient,wu_2018_application,Sarshar_2017_numerical}.

Stability analysis for W-methods with arbitrary matrices is still an open problem.  Stability questions for the case when $\Vert \J_n - \A_n \Vert$ is small are considered in \cite[Section IV.11]{Hairer_book_II}.  Similarly, exponential schemes are trivially $A$-stable when the $\varphi(\J_n)$ functions are evaluated exactly, however when these functions are evaluated using approximate methods this property is lost.  Quantifying the impact of the approximation of $\varphi(\J_n)$ on the stability of exponential methods is closely related to the stability questions for W-methods and also remains open.

K-methods are similarly plagued with stability questions, though have the benefit of making use of a very specific formulation for the approximate Jacobian $\A_n$.  In \cite{Buttner_1995_partitioning}, stability for W-methods making use of the Krylov based approximation, similar to that used for K-methods, is examined through the lens of contractivity for semilinear problems.  Here we examine the linear stability of K-methods and present two practical approaches for improving it.

The remainder of the paper is organized as follows. Section \ref{Radv:sec:kmethods} provides a brief overview of K-methods. Section \ref{Radv:sec:stability} performs a linear stability analysis for K-methods, and  Section \ref{Radv:sec:controlresid} introduces two new adaptive strategies for improving the stability and efficiency  through the direct control of stage residuals. Numerical results are presented in Section \ref{Radv:sec:numerical}, and concluding remarks are given in Section \ref{Radv:sec:conclusions}.

\section{K-type Time Integration Methods}
\label{Radv:sec:kmethods}

Rosenbrock-Krylov methods have the same form as Rosenbrock-W methods \eqref{Radv:eqn:genformW}, but use a particular approximation $\mathbf{A}_n$ of $\mathbf{J}_n$. Without loss of generality we restrict the presentation to the case of autonomous systems. Let $f_n:=f(y_n)$ and consider the $M$-dimensional Krylov space:
\begin{eqnarray}
\label{Radv:eqn:krylovspace}
	\K_M(\mathbf{J}_n\,,f_n) &=& \textnormal{span}\left\{\, f_n, \, \mathbf{J}_n\,f_n, \, \mathbf{J}_n^2\,f_n, \dots, \mathbf{J}_n^{M-1}\,f_n\, \right\} = \textnormal{span}\left\{v_1, v_2,  \dots, v_M\right\} \,.
\end{eqnarray}
An orthogonal basis $\{ v_i\}_{i=1,\dots,M}$ for $\K_M$ is constructed using a modified Arnoldi process \cite{vanderVorst}. The Arnoldi iteration returns two valuable pieces of information: a matrix $\mathbf{V}$ whose columns are the orthonormal basis vectors of $\K_M$, and an upper Hessenberg matrix $\mathbf{H}$, such that 
\begin{equation}
\label{Radv:eqn:htoj}
\mathbf{V} = \bigl[ v_1,\dots v_M \bigr] \in \Re^{N \times M}\,; \quad  
	\mathbf{H} = \mathbf{V}^T \, \mathbf{J}_n \, \mathbf{V} \in \Re^{M \times M}\,.
\end{equation}
{\it The Krylov approximation matrix $\mathbf{A}_n$ is the restriction of the full ODE Jacobian to the Krylov space:}
\begin{equation}
\label{Radv:eqn:ROK-approximation}
	\mathbf{A}_n = \mathbf{V}\, \mathbf{H}\,  \mathbf{V}^T =  \mathbf{V}\, \mathbf{V}^T\, \mathbf{J}_n\,  \mathbf{V}\,  \mathbf{V}^T\,.
\end{equation}

A Rosenbrock-Krylov method leverages the Krylov approximation matrix \eqref{Radv:eqn:ROK-approximation}, within the context of a Rosenbrock-W framework \cite{Tranquilli_2014_ROK}.  Similarly, an exponential-Krylov method leverages the Krylov approximation matrix \eqref{Radv:eqn:ROK-approximation} within the exponential-W framework \cite{Tranquilli_2014_expK,Sandu_2017_EPIRK}.  The special structure of the Krylov approximation matrix leads to the reduced space form of the method as summarized in Algorithm \ref{Radv:alg:K-autonomous-step}.  The Krylov approximation matrix has the additional feature that $\A_n^k f_n = \J_n^k f_n$ when $k$ is smaller than the dimension of the Krylov subspace ($k \leq m$), leading to a significant reduction in the number of order conditions, and thus the number of required stages, compared to a W-type method.  A full presentation of the order condition theory for K-methods can be found in \cite{Tranquilli_2014_ROK, Tranquilli_2014_expK,Sandu_2017_EPIRK}.

\begin{algorithm}[H]
\caption{One step of an autonomous K-type integrator \cite{Tranquilli_2014_ROK,Tranquilli_2014_expK}}\label{Radv:alg:K-autonomous-step}
\begin{algorithmic}[1]
   \State Compute $\Hb$ and $\V$ using the $N$-dimensional Arnoldi process \cite{vanderVorst}
   \For{$i = 1,\dots,s$}\Comment{For each stage, in succession}
      \begin{eqnarray*}
       F_i &=& f\left( y_n + \displaystyle\sum_{j=1}^{i-1}\alpha_{i,j}k_j\right) \\
       \psi_i &=& \V^T\, F_i \\
      \lambda_i &=& \varphi\left(h\gamma \Hb\right)\, \left( h\psi_i + h \Hb \displaystyle\sum_{j=1}^{i-1} \gamma_{i,j} \lambda_j\right) \\
       k_i &=& \V\, \lambda_i + h\, (F_i - \V\, \psi_i)
     \end{eqnarray*}
   \EndFor
   \State $y_{n+1} = y_n + \displaystyle\sum_{i=1}^s b_i k_i$
\end{algorithmic}
\end{algorithm}
%

\section{Stability Analysis}
\label{Radv:sec:stability}

The standard approach for investigating the linear stability of a time integration scheme is through application of the method to the linear Dahlquist test equation so that a local recurrence relation for the solution at $t_{n+1}$ in terms of the solution at $t_n$ can be constructed:
\[
y' = \lambda\, y, \quad y_0 = 1 \quad \Rightarrow \quad y_{n+1} = R(h\lambda)\,y_n.
\]
The set
\begin{equation*}
S = \left\{ z \in \mathbb{C}; \left|R(z)\right|\leq 1\right\}
\end{equation*}
is called the stability region of the method \cite{Hairer_book_II}.

Unfortunately, this approach is not sufficient for examining the stability of Rosenbrock methods where $\A_n \neq \J_n$, such as W-methods and K-methods, or for exponential schemes where the $\varphi(\J_n)$ are computed approximately.  For these schemes, since in general $\J_n$ and $\A_n$ are not simultaneously diagonalizable, it is necessary to consider the vector-valued linear test problem:
\begin{equation}
\label{Radv:eqn:lineartestproblem}
y' = \J\, y, \quad y_0 = \one.
\end{equation}
Application of a W- or K-method to equation \eqref{Radv:eqn:lineartestproblem} leads to the local recurrence
\begin{equation*}
y_{n+1} = R(h\,\J, h\A)\,y_n\,,
\end{equation*}
so that the stability region is determined by those step sizes $h$ for which the spectral radius of the transfer matrix is bounded: $\rho(R) \le 1$.


Specifically, the $s$-stage Rosenbrock method \eqref{Radv:eqn:genformW} applied to the linear test problem \eqref{Radv:eqn:lineartestproblem} reads:
\begin{subequations}
\label{Radv:eqn:linearros}
\begin{eqnarray}
k_i & = & h\,\J \left(y_0 + \displaystyle\sum_{j=1}^{i-1} \alpha_{i,j} k_j\right) + \left(h\A \displaystyle\sum_{j=1}^{i} \gamma_{i,j}k_j \right), ~~ i=1,\dots,s, \\
y_{n+1} & = & y_n + \displaystyle\sum_{i=1}^s b_i k_i.
\end{eqnarray}
\end{subequations}

We define the coefficient matrices: 
\begin{equation*}
\alphaB = \begin{bmatrix}
	 0 & \cdots & \cdots & 0 \\
	 \alpha_{2,1} & 0 & \dots & \vdots \\
	 \vdots & \ddots & \ddots & \vdots \\
	 \alpha_{s,1} & \cdots & \alpha_{s,s-1} & 0
	 \end{bmatrix}, \quad
\gammaB = \begin{bmatrix}
	\gamma & 0 & \cdots & 0 \\
	\gamma_{2,1} & \gamma & 0 & \vdots \\
	\vdots & \ddots & \ddots & 0 \\
	\gamma_{s,1} & \cdots & \gamma_{s,s-1} & \gamma
	\end{bmatrix}, \quad
\betaB = \alphaB + \gammaB,
\end{equation*}
and the supervectors:
\begin{equation*}
\Kb_n = \begin{bmatrix} k_1 \\ \vdots \\ k_s \end{bmatrix}\in \Re^{N s}, \quad 
\Yb_n = \begin{bmatrix} y_n \\ \vdots \\ y_n \end{bmatrix} = \left(\one_{s} \otimes \I_{N} \right)\,y_n \in \Re^{N s}. 
\end{equation*}

The method \eqref{Radv:eqn:linearros} can be written in matrix notation as follows:
\begin{subequations}
\label{eqn:ROS-w-approximate-Jacobian}
\begin{eqnarray}
\label{eqn:ROS-w-approximate-Jacobian-s1}
\Kb_n & = & h\,\left(\I_s \otimes \J\right)\Yb_n + h\,\bigl[\left(\alphaB \otimes \J\right) + \left(\gammaB \otimes \A\right)\bigr]\Kb_n\\
\label{eqn:ROS-w-approximate-Jacobian-s2}
& = & h\,\left(\I_s \otimes \J\right)\Yb_n + h\,\left(\betaB \otimes \J\right)\Kb_n + \Rb_n, \\
y_{n+1} & = & y_n + \left(b^T \otimes \I_N\right)\, \Kb_n.
\end{eqnarray}
\end{subequations}
The stage equation \eqref{eqn:ROS-w-approximate-Jacobian-s2} makes exclusive use of the exact Jacobian $\J$ by adding the residual supervector:
\begin{equation}
\label{eqn:stage-residuals}
\Rb_n := h\,\gammaB \otimes (\A-\J)\,\Kb_n = \begin{bmatrix} r_{1} \\ \vdots \\ r_{s} \end{bmatrix} \in \Re^{N s},
\end{equation}
where $r_{i}$ is the $i$-th stage residual vector, to obtain the same intermediate stage values $\Kb_n$.

Isolating $\Kb_n$ in both \eqref{eqn:ROS-w-approximate-Jacobian-s1} and \eqref{eqn:ROS-w-approximate-Jacobian-s2}, leads to
\begin{eqnarray*}
\Kb_n & = & h\bigl[\I_{Ns} - \left(\alphaB \otimes h\,\J\right) - \left(\gammaB \otimes h\A\right)\bigr]^{-1}\left(\I_s\otimes\J\right)\Yb_n   \\
      & = &  h\,\left(\I_{Ns} - h\,\betaB\otimes\J\right)^{-1}\left(\I_s\otimes\J\right)\Yb_n + \left(\I_{Ns} - h\,\betaB\otimes\J\right)^{-1} \Rb_n.
\end{eqnarray*}
From here we construct an explicit formulation of the residual vector $\Rb_n$, and apply the useful linear algebra identity:\
\begin{equation}
\label{eqn:la-identity}
\begin{split}
& \bigl[\I_{Ns} - \left(\alphaB \otimes h\,\J\right) - \left(\gammaB \otimes h\A\right)\bigr]^{-1} - 
\bigl[\I_{Ns} - \left(\betaB \otimes h\,\J\right)\bigr]^{-1}\\
&= - \bigl[\I_{Ns} - \left(\betaB \otimes h\,\J\right)\bigr]^{-1}\, \bigl[\gammaB \otimes \left( h\,\J-h\A\right)\bigr]
\,
\bigl[\I_{Ns} - \left(\alphaB \otimes h\,\J\right) - \left(\gammaB \otimes h\A\right)\bigr]^{-1},
\end{split}
\end{equation}
to obtain
\begin{eqnarray*}
\Rb_n & = &
\left(\I_{Ns} - h\betaB\otimes\J\right)\,\bigl[\I_{Ns} - \left(\alphaB \otimes \J\right) - \left(\gammaB \otimes \A\right)\bigr]^{-1}\left(\I_s\otimes h \J\right)\Yb_n 
     -  \,\left(\I_s\otimes h \J\right)\Yb_n \\
\left(\I_{Ns} - \betaB\otimes h\,\J\right)^{-1} \Rb_n &=& 
     \left(\bigl[\I_{Ns} - \left(\alphaB \otimes \J\right) - \left(\gammaB \otimes \A\right)\bigr]^{-1} - \left(\I_{Ns} - h\betaB\otimes \J\right)^{-1}\right)\left(\I_s\otimes h\,\J\right)\Yb_n.\\
\end{eqnarray*}
The stability of the Rosenbrock-K method is then given by:
\begin{eqnarray*}
y_{n+1} & = & y_n + \left(b^T \otimes \I\right) \Kb_n \\
 & = & \left( \I_{Ns} + \left(b^T \otimes \I_N \right) \left(\I_{Ns} - \betaB\otimes h\,\J\right)^{-1}\left(\one_s \otimes h\,\J\right)\right)y_n + \left(b^T \otimes \I_N \right)\left(\I_{Ns} - \betaB\otimes h\,\J\right)^{-1} \Rb_n \\
& = & R\left(h\,\J\right)\, y_n + \left(b^T \otimes \I_N \right)\left(\bigl[\I_{Ns} - \left(\alphaB \otimes \J\right) - \left(\gammaB \otimes \A\right)\bigr]^{-1} - \left(\I_{Ns} - h\betaB\otimes \J\right)^{-1}\right)\left(\one_s\otimes h\,\J\right) y_n\\
& = & R\left(h\,\J\right)\,y_n + S\left(h\,\J, h\A\right)\,y_n\\
& = & \widetilde{R}\left(h\,\J,h\A\right)\, y_n.
\end{eqnarray*}
Here $R\left(h\,\J\right)$ is the stability function of the traditional Rosenbrock method with the same coefficients.
The effective stability function $\widetilde{R}\left(h\,\J,h\A\right)$ of the Rosenbrock-K method is the Rosenbrock stability function $R$ plus an additional term that depends on both $\J$ and $\A$. This additional term is the stage stability function $S\left(h\,\J, h\A\right)$.

The linear stability of the underlying Rosenbrock method does not guarantee the stability of the Rosenbrock-Krylov method
\[
\rho\big( R\left(h\,\J\right) \big) \le 1
\]
does not guarantee the stability of the Rosenbrock-Krylov method
\[
\rho\big( \widetilde{R}\left(h\,\J,h\A\right) \big) \le 1,
\]
unless the stage stability function $S(h\,\J,h\A)$ is small in some sense.  This function can be expressed explicitly as:
\[
\begin{split}
S\left(h\,\J, h\A\right)\,y_n &= -\left(b^T \otimes \I_N \right) 
\bigl[\I_{Ns} - \left(\betaB \otimes h\,\J\right)\bigr]^{-1}\, \bigl[\gammaB \otimes \left( h\,\J-h\A\right)\bigr]
\,
\bigl[\I_{Ns} - \left(\alphaB \otimes h\,\J\right) - \left(\gammaB \otimes h\A\right)\bigr]^{-1}
\left(\one_s\otimes h\,\J\right) y_n \\
&= -\left(b^T \otimes \I_N \right) 
\bigl[\I_{Ns} - \left(\betaB \otimes h\,\J\right)\bigr]^{-1}\, \bigl[\gammaB \otimes \left( h\,\J-h\A\right)\bigr]
\,
\Kb_n,
\end{split}
\]
or more generally as:
\begin{equation}
\label{eqn:stage-resid}
\begin{split}
S\left(h\,\J, h\A\right)\,y_n 
 & = -\left(b^T \otimes \I_N \right) 
\bigl[\I_{Ns} - \left(\betaB \otimes h\,\J\right)\bigr]^{-1}\, \bigl[\gammaB \otimes \left( h\,\J-h\A\right)\bigr]
\,
\bigl[\I_{Ns} - \left(\gammaB \otimes h\A\right)\bigr]^{-1}\,\mathbf{F}_n,
\end{split}
\end{equation}
where $\mathbf{F}_n$ are the stage function values \eqref{Radv:eqn:genformWrhs} concatenated for all stages concatenated:
\[
\mathbf{F}_n = \begin{bmatrix} F_1 \\ \vdots \\ F_s \end{bmatrix}.
\]
In later sections we will attempt to keep \eqref{eqn:stage-resid} under control through two different strategies.  The most obvious of these approaches, presented in Section \ref{Radv:sec:controlresid}, directly controls the residual $\Rb_n$ of the linear system solutions.  This can be accomplished by simply increasing the dimension of the underlying space on which the Krylov approximation matrix $\A$ is based, and has the effect of decreasing the projection error
\begin{equation*}
\Vert \J - \A \Vert = \Vert \J - \V\,\V^T\J\,\V\,\V^T\Vert,
\end{equation*}
and directly reducing the corresponding term in \eqref{eqn:stage-resid}.  The alternative approach, based on the specific form of the residuals presented later, attempts to control \eqref{eqn:stage-resid} by extending the Krylov space with the right hand side values at each internal stage \eqref{Radv:eqn:genformWrhs}.

%

%
\section{Adaptive Methods to Increase Rosenbrock-K Stability}
\label{Radv:sec:adaptive}

\subsection{Controlling the residual magnitude}
\label{Radv:sec:controlresid}

The most direct strategy for improving the stability of a Rosenbrock-Krylov method is by  controlling the residual magnitude at each stage of the time integration scheme.
\begin{theorem}[Stage residual of a Rosenbrock-Krylov method]
\label{Radv:thm:istageresidual}
When the time integration scheme \eqref{Radv:eqn:genformW} is implemented as in Algorithm \ref{Radv:alg:K-autonomous-step}, the residual of the $i^{\rm th}$ stage is given by:
\begin{equation}
\label{Radv:eqn:ithstageresid}
r_{M;i} = -\displaystyle\sum_{j=1}^{i} h^2\,\gamma_{i,j}\,\J_n\,\left(F_j - \V_M\psi_j\right) - h\, \h_{M+1,M}\,v_{M+1}\,e_M^T\displaystyle\sum_{j=1}^{i}\gamma_{i,j}\lambda_j,
\end{equation}
where $r_{M;i} := r_{i}$ in \eqref{eqn:stage-residuals} with the dimension of the underlying Krylov subspace $M = dim(\K_M)$ made explicit. 
\end{theorem}
\begin{proof}
We begin with the formulation of the method in Algorithm \ref{Radv:alg:K-autonomous-step} as:
\begin{subequations}
\label{Radv:eqn:rokresid}
\begin{eqnarray}
\label{Radv:eqn:rokresidfullstage}
k_i & = & hF_i + h\,\J_n\displaystyle\sum_{j=1}^{i}\gamma_{i,j}k_j + r_{M;i} \\
\label{Radv:eqn:rokresidredstage}
\lambda_i & = & h\psi_i + h\Hb
\displaystyle\sum_{j=1}^{i}\gamma_{i,j}\lambda_j \\
\label{Radv:eqn:rokresidpsi}
\psi_i & = & \V^T\,F_i \\
\label{Radv:eqn:rokresidk}
k_i & = & \V\,\lambda_i + h\,\left(F_i - \V\,\psi_i\right).
\end{eqnarray}
\end{subequations}
Inserting \eqref{Radv:eqn:rokresidk} into \eqref{Radv:eqn:rokresidfullstage} and expanding leads to:
\begin{equation*}
r_{M;i} = \V\,\lambda_i - h\,\V\,\psi_i - h^2\J_n\displaystyle\sum_{j=1}^i\gamma_{i,j} \left(F_j - \V\,\psi_j\right) - h\,\J_n\,\V\displaystyle\sum_{j=1}^i\gamma_{i,j}\lambda_j.
\end{equation*}
Next, apply the Arnoldi relationship
\begin{equation}
\label{Radv:eqn:arnoldirelation}
\J_n\,\V = \V\,\Hb + \h_{M+1,M}\,v_{M+1}\,e_M^T,
\end{equation}
to the terms containing $\lambda_j$:
\begin{equation*}
r_{M;i} = \V\,\lambda_i - h\,\V\,\psi_i - h^2\J_n\displaystyle\sum_{j=1}^i \gamma_{i,j}\left(F_j - \V\,\psi_j\right) - h\,\V\,\Hb\displaystyle\sum_{j=1}^i\gamma_{i,j}\lambda_j - \h_{M+1,M}v_{M+1}e_M^T\displaystyle\sum_{j=1}^i \gamma_{i,j}\lambda_j
\end{equation*}
Collecting terms with \eqref{Radv:eqn:rokresidredstage} in mind, we obtain
\begin{equation*}
r_{M;i} = \V\underbrace{\left(\lambda_i - h\psi_i - h\Hb\displaystyle\sum_{j=1}^i\gamma_{i,j}\lambda_j\right)}_{=0} -  h^2\J_n\displaystyle\sum_{j=1}^i \gamma_{i,j}\left(F_j - \V\,\psi_j\right)- \h_{M+1,M}v_{M+1}e_M^T\displaystyle\sum_{j=1}^i \gamma_{i,j}\lambda_j,
\end{equation*}
and arrive at the final result \eqref{Radv:eqn:ithstageresid}.
 \qed
\end{proof}

Unfortunately, the residuals depend not only on the reduced stage solutions $\lambda_j$, but also on the full space matrix-vector products $\J_n\left(F_j-\V\,\psi_j\right)$.  The dependence on $\lambda_j$'s means that verifying the size of the residual in the final stage requires essentially computing the full space timestep, and so any strategy attempting to bound this residual will be computationally infeasible.  An alternative approach is to bound the residuals only in the first stage, so that only $\lambda_1$ is required. This has the computational benefit that $\left(F_1 - \V\,\psi_1\right) = 0$ since $F_1 \in \K_M$, and so no full space matrix-vector products are required.

\begin{corollary}[Residual of the first stage of a Rosenbrock-Krylov method]
When the time integration scheme \eqref{Radv:eqn:genformW} is implemented as in Algorithm \ref{Radv:alg:K-autonomous-step}, then the residual of the first stage is given by:
\begin{equation*}
r_{M;1} = -h\,\gamma\,\h_{M+1,M}\,v_{M+1}\,e_M^T\,\lambda_1,
\end{equation*}
with norm
\begin{equation*}
\Vert r_{M;1} \Vert = \left| \,h\,\gamma\,\h_{M+1,M}\,\right|\cdot\left|\,e_M^T\,\lambda_1 \, \right|.
\end{equation*}
\end{corollary}

It is then possible to implement a Rosenbrock-Krylov method which automatically determines the dimension $M$ of the Krylov subspace such as to control the size of the residual in the first stage through the use of an adaptive Arnoldi process. This method is summaruzed in Algorithm \ref{Radv:alg:adaptive-arnoldi}. The computational cost of computing the residuals can be reduced by only computing them for a subset of iterations.  For example it was proposed in \cite{Hochbruck_1998_exp} to compute the residuals only for the test indices $i \in \left\{1,2,3,4,6,8,11,15,20,27,36,48\right\}$.
\begin{algorithm}[h]
\caption{Modified Arnoldi iteration \cite{Saad}}\label{Radv:alg:adaptive-arnoldi}
\begin{algorithmic}[l]
   \State $\beta = \left\Vert f_n \right\Vert~$, $\quad  \mathbf{V}_1 = f_n/\beta~$.
   \For{$i=1,\dots,M$}
       \State $\displaystyle \zeta =  \mathbf{J}_n \, v_i$  
      \For{$j = 1,\dots,i$} 
         \State $\h_{j,i} = \left\langle \zeta , \mathbf{V}_j \right\rangle$
         \State $\displaystyle \zeta = \zeta  - \h_{j,i}\, \mathbf{V}_j$
      \EndFor
      \State $\displaystyle \h_{i+1,i} = \left\Vert \zeta \right\Vert$
      \State $\displaystyle \mathbf{V}_{i+1} = \zeta/\h_{i+1,i}$
      \If{ $i \in $ testIndices }
	      \State $\lambda_1 = \left(\I - h\gamma\Hb_i\right)^{-1}f_n$
     	 \If{$\left|h\gamma\h_{i+1,i}\right|\left|e_M^T\lambda_1\right| \leq TOL$}
      		\State \textbf{return}
	 \EndIf
      \EndIf
   \EndFor
\end{algorithmic}
\label{Radv:alg:arnoldi}
\end{algorithm}
%

\subsection{Adding vectors from outside the Krylov space}
\label{Radv:sec:outsidebasis}

An alternative to increasing the size of the Krylov subspace in the very beginning is to extend the subspace with vectors from other sources, not necessarily from the same Arnoldi sequence.  

Consider the Krylov space $\K_M$ with basis $\V_M$ and the upper Hessenberg matrix $\Hb_M$.
In order to extend the Krylov basis with arbitrary vectors $b_1, \dots, b_r$, one forms an extended basis set $\V_{M+r}$ through the same orthonormalization procedure present in the Arnoldi iteration:
\begin{equation*}
\V_{M+r} = \begin{bmatrix} v_1, v_2, \dots, v_M, \bar{v}_1, \dots, \bar{v}_r\end{bmatrix},
\quad
\V_{M+r}^T\,\V_{M+r}  = \I_{M+r}.
\end{equation*}
The extended matrix $\Hb_{M+r} \in \Re^{(m+r)\times(m+r)}$ is constructed such as to maintain the relationship $\Hb_{M+r} = \V_{M+r}^T\,\J_n\, \V_{M+r}$. We have that:
\begin{equation*}
\Hb_{M+r} = \begin{bmatrix} \begin{bmatrix} \Hb_M \\ 0_{r \times 1} \end{bmatrix}
& h_{M+1} & \dots & h_{M+r} \end{bmatrix},
\end{equation*}
with the additional columns given by:
\begin{equation*}
h_{M+i} = \V_{M+r}^T\,\left(\J_n\, \bar{v}_i\right) \quad  \textrm{for } i = 1,\dots,r.
\end{equation*}

\begin{remark}
In previous sections, the Arnoldi relationship \eqref{Radv:eqn:arnoldirelation} is written in terms of $v_{M+1}$, the next vector which would be added to the Krylov basis $\V_M$, and $\h_{M+1,M}$ a coefficient corresponding to $v_{M+1}$ which would be added to the upper-Hessenberg matrix $\Hb_M$. 
Using the definition of the Arnoldi process we have that
\begin{equation*}
\h_{M+1,M}\,v_{M+1} = \left(\I - \V_M\V_M^T\right)\,\J_n\, v_M,
\end{equation*}
where $v_M$ is the last column in the Krylov basis matrix $\V_M$. With this, one can rewrite the Arnoldi relationship \eqref{Radv:eqn:arnoldirelation} as follows:
\begin{equation}
\label{Radv:eqn:arnoldirelationalt}
\J_n\, \V_M = \V_M\, \Hb_M + \left(\I - \V_M\,\V_M^T\right)\,\J_n\, v_M\, e_M^T.
\end{equation}
\end{remark}

\begin{lemma}[Arnoldi-like relationship for extended basis]
\label{Radv:lem:arnoldirelationext}
When the Krylov basis matrix $\V_M$ is extended with orthonormal (but otherwise arbitrary) vectors $\bar{v}_1$, $\bar{v}_2$, $...$, $\bar{v}_r$, the Arnoldi relationship \eqref{Radv:eqn:arnoldirelationalt} extendeds to the new basis as follows:
\begin{equation}
\label{Radv:eqn:arnoldirelationext}
\J_n \, \V_{M+r} = \V_{M+r} \Hb_{M+r} + \left(\I - \V_M\,\V_M^T\right)\,\J_n\, v_M\, e_M^T + \left(\I - \V_{M+r} \V_{M+r}^T\right)\,\J_n\, \sum_{i=1}^r \bar{v}_i\, e_{M+i}^T,
\end{equation}
where $e_{M+r,j}$ is the $j$'th canonical basis vector of $\Re^{(M+r)}$. 
\end{lemma}
\begin{proof}
Recall that $\V_{M+r}$ and $\Hb_{M+r}$ have the following structure:
\begin{equation*}
\V_{M+r} = \begin{bmatrix} \V_M & \overline{\V} \end{bmatrix} \in \Re^{N \times (M+r)}, \quad \Hb_{M+r} = 
\begin{bmatrix} \begin{bmatrix} \Hb_M \\ 0 \end{bmatrix} & \overline{\Hb} \end{bmatrix} \in \Re^{(M+r) \times (M+r)},
\end{equation*}
where $\overline{\V} = \begin{bmatrix} \bar{v}_1, \dots, \bar{v}_r \end{bmatrix}$ and $\overline{\Hb} = \V_{M+r}^T \left(\J \overline{\V}\right)$. Computing $\J \V_{M+r}$ gives
\begin{equation*}
\J \V_{M+r} = \bigl[\begin{array}{c|c} \J \V_M & \J \overline{\V} \end{array}\bigr] = \bigl[\begin{array}{c|c} \V_M \Hb_M + \left(\I - \V_M\,\V_M^T\right)\J v_M e_{M;M}^T & \J \overline{\V} \end{array}\bigr],
\end{equation*}
which follows from the Arnoldi relationship (with $e_{M,M} \in \Re^{M}$). Similarly, computing $\V_{M+r} \Hb_{M+r}$ gives
\begin{equation*}
\V_{M+r} \Hb_{M+r} = \bigl[\begin{array}{c|c} \V_M & \overline{\V} \end{array}\bigr] \bigl[\begin{array}{c|c} \Hb_M & \multirow{2}{*}{$\overline{\Hb}$} \\ 0 & \end{array}\bigr] = \bigl[\begin{array}{c|c} \V_M \Hb_M & \V_{M+r} \overline{\Hb} \end{array}\bigr].
\end{equation*}
Subtracting yields:
\begin{equation*}
\J \V_{M+r} - \V_{M+r} \Hb_{M+r} = \bigl[\begin{array}{c|c} \left(\I - \V_M\,\V_M^T\right)\J v_M e_{M;M}^T & \J \overline{\V} - \V_{M+r} \overline{\Hb} \end{array}\bigr].
\end{equation*}
Substituting in the definition of $\overline{\Hb}$ and rearranging leads to:
\begin{equation*}
\J \V_{M+r} = \V_{M+r} \Hb_{M+r} + \bigl[\begin{array}{c|c} \left(\I - \V_M\,\V_M^T\right)\J v_M e_{M;M}^T & \left(\I - \V_{M+r} \V_{M+r}^T \right)\J \overline{\V} \end{array}\bigr].
\end{equation*}
Finally, we rewrite the last matrix as a sum of rank-1 matrices using canonical basis vectors $e_{M+r;j} \in \Re^{(M+r)}$:
\begin{eqnarray*}
&&\bigl[\begin{array}{c|c} \left(\I - \V_M\,\V_M^T\right)\J v_M e_{M;M}^T & \left(\I - \V_{M+r} \V_{M+r}^T \right)\J \overline{\V} \end{array}\bigr] \\
&& \qquad = \left(\I - \V_M\,\V_M^T\right)\J v_M e_{M+r;M}^T + \sum_{i = 1}^r \left(\I - \V_{M+r}\V_{M+r}^T\right)\J\bar{v}_i e_{M+r;M+i}^T.
\end{eqnarray*}
With this substitution and after some rearranging we obtain equation \eqref{Radv:eqn:arnoldirelationext}. 
%
%
\qed
\end{proof}

With the ability to extend the basis $\V$ with arbitrary vectors, the question now is: what vectors shall we add? Intuitively, we seek to add those vectors that provide improved stability, or related, improved benefits in controlling the residuals of later stages of the Rosenbrock-Krylov method.

\begin{remark}[Extending the Krylov space]
When the Krylov space $\V_{M+i}$ changes for each stage, so does the Jacobian projection 
\[
\A_i = \V_{M+i}\V_{M+i}^T\,\J_n\,\V_{M+i}\V_{M+i}^T.
\]
In this case the additional stability term defined by equation \eqref{eqn:stage-resid} changes to:
\begin{equation}
\label{eqn:stage-resid-mod}
\begin{split}
S\left(h\,\J, h\A_1,\dots,\h\A_s\right)\,y_n
&= -\left(b^T \otimes \I \right) 
\bigl[\I - \left(\betaB \otimes h\,\J\right)\bigr]^{-1}\, \bigl[\gammaB \otimes h\J-h\widetilde{\A}_{\gammaB} \bigr]
\,
\bigl[\I - h\,\widetilde{\A}_{\gammaB} \bigr]^{-1}\,\mathbf{F},
\end{split}
\end{equation}
where the block matrix $\gammaB \otimes \A$ in \eqref{eqn:stage-resid} has been replaced by:
\[
\widetilde{\A}_{\gammaB} = \begin{bmatrix} \gamma_{i,j}\, \A_j \end{bmatrix}_{1 \le i,j \le s}.
\]
Since the coefficient matrices $\betaB$ and $\gammaB$ are lower triangular, the matrices $\betaB \otimes h\,\J$,
$\gammaB \otimes h\J$, as well as $h\widetilde{\A}_{\gammaB}$ are all block lower triangular. The $i$-th stage component vector in \eqref{eqn:stage-resid-mod} has the form:
\[
\begin{split}
&\left( 
\bigl[\I - \left(\betaB \otimes h\,\J\right)\bigr]^{-1}\,
\bigl[h\,\gammaB \otimes \J-h\,\widetilde{\A}_{\gammaB} \bigr]\,
\bigl[\I - h\,\widetilde{\A}_{\gammaB} \bigr]^{-1}\,\mathbf{F} \right)_i \\
& \qquad = 
\bigl[\I - h\,\gamma\,\J \bigr]^{-1}\,
\bigl[h\,\gamma\, \J-h\,\gamma \A_i \bigr]\,
\bigl[\I - h\,\gamma\,\A_i \bigr]^{-1}\,F_i + [\dots] \\
& \qquad = 
\bigl[\I - h\,\gamma\,\J \bigr]^{-1}\,\,F_i
-
\bigl[\I - h\,\gamma\,\A_i \bigr]^{-1}\,F_i + [\dots],
\end{split}
\]
where the first part corresponds to the action of the diagonal block, and the remaining part $[\dots]$ contains linear combinations of off-diagonal blocks times the previous stage function values $F_1, \dots, F_{i-1}$.  The first part is the error in the solution of the stage linear system $\bigl[\I - h\,\gamma\,\J \bigr]\,X_i = F_i$  when solved in the subspace space spanned by $\V_{M+i}$ (with $\J$ replaced by $\A_i$). A strategy to reduce $S\left(h\,\J, h\A_1,\dots,\h\A_s\right)\,y_n$ is to extend the Krylov space at stage $i$ such as to reduce the linear system solution error. In the next section we discuss extending the subspace with the right hand side of the system $F_i$. Note that the off-diagonal components $[\dots]$ components  may remain large even after extending the space, as this reduces the first part of the vector $S\left(h\,\J, h\A_1,\dots,\h\A_s\right)\,y_n$.
\end{remark}


\subsection{Exploiting the form of the residual}
\label{sec:ext-space}

We see from equation \eqref{Radv:eqn:ithstageresid} in Theorem \ref{Radv:thm:istageresidual} that residuals for stages beyond the first contain terms for which it is difficult to guarantee boundedness, even when the first stage residuals are controlled.  We attempt to overcome this fact, by extending the basis $\V_M$ at the $i$'th stage to include the intermediate RHS values $F_j$ for $j = 2, \dots, i$ so that  
%
\begin{equation*}
\left(F_i - \V_{M;i}\psi_i\right) = 0, \quad \textrm{for } i = 1, \dots, s.
\end{equation*}

The matrices $\V_{M;i}$ and $\Hb_{M;i}$ can be constructed iteratively, extending the basis $\V_{M;i-1}$ with $F_i$ at each stage to form $\V_{M;i}$ and $\Hb_{M;i}$ as outlined in the previous section.  One difficulty of extending the basis in this way, is that the $\Lb\Ub$-decomposition of the left-hand-side of \eqref{Radv:eqn:genformWstage}, $\left(\I - h\gamma\Hb\right)$ cannot be reused.  This problem can be overcome through the use of an update to the $\Lb$ and $\Ub$ matrices.  The $\Lb\Ub$-decomposition is constructed to solve
\begin{equation*}
\begin{array}{rcl}
\left(\I - h\gamma\Hb\right) x & = & b \\
\Pb\left(\I - h\gamma\Hb\right) x & = & \Pb b \\
\Lb \Ub x & = & \Pb b \\
\end{array}
\end{equation*}
so that
\begin{equation*}
\begin{array}{rcl}
\Pb\left(\I - h\gamma\Hb\right) & = & \Lb\Ub \\
\Ub & = & \Lb^{-1}\Pb\left(\I - h\gamma\Hb\right)
\end{array}
\end{equation*}
and so exploiting the upper-Hessenberg structure of $\Hb$, the $\Lb\Ub$-decomposition can be updated as
\begin{equation*}
\begin{bmatrix}\mathbf{u}_{1,M+1} \\ \vdots \\ \mathbf{u}_{M,M+1} \end{bmatrix} = -h\gamma\Lb_M^{-1}\Pb_M  \begin{bmatrix}\h_{1,M+1} \\ \vdots \\ \h_{M,M+1} \end{bmatrix}, \quad \mathbf{u}_{M+1,M+1} = \left(1 - h\gamma\h_{M+1,M+1}\right), \quad \mathbf{l}_{M+1,M+1} = 1
\end{equation*}

\begin{theorem}[Residual of the $i$th stage with extended basis]
\label{Radv:thm:sstageresidualext}
When the time integration scheme \eqref{Radv:eqn:genformW} is implemented with the basis extended (via the process described in Section \ref{Radv:sec:outsidebasis}) at each stage to include the current right-hand-side vector, $F_i$, then the residual of the $i$th stage when using an $m$-dimensional Krylov space is given by:
\begin{equation}
\label{Radv:eqn:sresidext}
r_{M;i} = - h\,\left(\I - \V_M\,\V_M^T\right)\J v_M e_{M+i-1,M}^T \sum_{j=1}^i \gamma_{i,j} \widehat{\lambda}_j - h \left(\I - \V_{M;i}\V_{M;i}^T\right)\J \sum_{k=1}^{i-1} v_{M+k} e_{M+i-1,M+k}^T \sum_{j=1}^i \gamma_{i,j} \widehat{\lambda}_j,
\end{equation}
where $e_j \in \Re^{(M+i-1)}$ are canonical basis vectors, $\widehat{\lambda}_j = \begin{bmatrix} \lambda_j^T, 0, \dots, 0 \end{bmatrix}^T \in \Re^{(M+i-1)}$, and $\V_{M;1} = \V_M$.
\end{theorem}
\begin{proof}
The general stage equations for the method are:
\begin{equation*}
\begin{array}{rcl}
k_i & = & h F_i + h \J \displaystyle\sum_{j=1}^{i} \gamma_{i,j} k_j + r_{M;i} \\
\widehat{\lambda}_i & = & h \psi_i + h \Hb_{M;i} \displaystyle\sum_{j=1}^{i-1} \gamma_{i,j} \widehat{\lambda}_j \\
\psi_i & = & \V_{M;i}^T F_i \\
k_i & = & \V_{M;i} \widehat{\lambda}_i + h\,\left(F_i - \V_{M;i}\psi_i\right) = \V_{M;i}\widehat{\lambda}_i
\end{array}
\end{equation*}
Additionally, we define $\widehat{\lambda}_j \in \Re^{(M+i-1)}$  as follows:
\begin{equation*}
\widehat{\lambda}_j = \left\{\begin{array}{cl}
\begin{bmatrix} \lambda_j^T, 0, \dots, 0 \end{bmatrix}^T, & j < i, \\
\lambda_j, & j = i,
\end{array}\right.
\end{equation*}
where reduced space solutions $\lambda_j \in \Re^{(M+j-1)}$ from previous stages are extended with zeros to match the dimension of the current stage system.

Starting from the full space equation, we make substitutions for $F_i$ and the $k$'s:
\begin{equation*}
r_{M;i} = \V_{M;i}\widehat{\lambda_i} - h\,\V_{M;i}\psi_i - h \J\V_{M;i} \displaystyle \sum_{j=1}^{i} \gamma_{i,j}  \widehat{\lambda}_j,
\end{equation*}
Making use of lemma \ref{Radv:lem:arnoldirelationext}, we make substitutions for appearances of $\J \V_{M;i}$:
\begin{multline*}
r_{M;i} = \V_{M;i}\widehat{\lambda}_i - h\,\V_{M;i}\psi_i - h\,\V_{M;i}\Hb_{M;i}\displaystyle\sum_{j=1}^i \gamma_{i,j}\widehat{\lambda}_j - h\left(\I - \V_M\V_M^T\right)\J v_Me_{M+i-1;M}^T\displaystyle\sum_{j=1}^i \gamma_{i,j}\widehat{\lambda}_j \\ 
- h\left(\I - \V_{M;i}\V_{M;i}^T\right)\J\displaystyle\sum_{k=1}^{i-1}v_{M+k}e_{M+i-1;M+k}^T \displaystyle\sum_{j=1}^i\gamma_{i,j}\widehat{\lambda}_j
\end{multline*}
Collecting terms and rearranging allows us to find and cancel the terms from the reduced space equation:
\begin{multline*}
r_{M;i} = \V_{M;i}\underbrace{\bigl[\widehat{\lambda}_i - h\psi_i - h\Hb_{M;i}\displaystyle\sum_{j=1}^{i}\gamma_{i,j}\widehat{\lambda_j}\bigr]}_{=0} - h\left(\I - \V_M\V_M^T\right)\J v_Me_{M+i-1;M}^T\displaystyle\sum_{j=1}^i \gamma_{i,j}\widehat{\lambda}_j \\ 
- h\left(\I - \V_{M;i}\V_{M;i}^T\right)\J\displaystyle\sum_{k=1}^{i-1}v_{M+k}e_{M+i-1;M+k}^T \displaystyle\sum_{j=1}^i\gamma_{i,j}\widehat{\lambda}_j
\end{multline*}
\begin{equation*}
r_{M;i} = - h \left(\I - \V_M\,\V_M^T\right)\J v_M e_{M+i-1;M}^T \sum_{j=1}^i \gamma_{i,j} \widehat{\lambda}_j - h \left(\I - \V_{M;i} \V_{M;i}^T\right)\J \sum_{k=1}^{i-1} v_{M+k} e_{M+i-1;M+k}^T \sum_{j=1}^i \gamma_{i,j} \widehat{\lambda}_j.
\end{equation*}
\qed
\end{proof}

The new residual coming from the extended basis \eqref{Radv:eqn:sresidext} appears to be a substantial improvement over the residuals from the vanilla method.  Since here the residuals consist exclusively of terms projected out of the span of $\V_M$ and $\V_{M;s}$.

\section{Numerical Results}
\label{Radv:sec:numerical}

In this section, we evaluate the behavior of the two presented modifications to the Rosenbrock-K method on a selection of test problems. As both the adaptive basis size modification and the basis extension variant are intended to increase the stability of the method, testing primarily focused on problems that demonstrate stiff behavior: CBM-IV as a stiff system of densely-coupled ODEs, and the Allen-Cahn equation as a PDE with variable stiffness based on coefficient selection. Also, the shallow water equations were tested in order to demonstrate behavior in a non-stiff case. All tests were performed in Matlab on a single workstation, using an implementation of the ROK method with an adaptive timestep-size error controller, and results were compared against reference solutions computed by Matlab's \texttt{ode15s} integrator. Figures are produced by sweeping over the same set of relative error tolerances $\{10^{-i}\}_{i=2...10}$.

All the figures in this section use the same labeling scheme, as follows. Data sets labeled $M$ use the stated fixed number of basis vectors, where those labeled $R$ use an adaptive number based on the stated residual tolerance. Dashed lines indicate basis extension and can apply to either fixed or adaptive basis size selection, indicated in the label by the suffix $ext$. Also, lines labeled $R = \text{tol}$ are experiments where the Arnoldi residual tolerance is set to the same value as the tolerance for the adaptive timestep-size error controller.

\subsection{Non-stiff case: shallow water equations}
First, we examine relative performance of the methods on the shallow water equations. The shallow water equations in spherical coordinates are
\begin{subequations}
\begin{eqnarray}
\label{Radv:eqn:shallowwater}
\frac{\partial u}{\partial t} + \frac{1}{a \cos\theta} \left(u\frac{\partial u}{\partial \lambda} + v \cos\theta\frac{\partial u}{\partial \theta}\right) - \left(f + \frac{u\tan\theta}{a}\right)a + \frac{g}{a\cos\theta}\frac{\partial h}{\partial \lambda} & = & 0 \\
\frac{\partial v}{\partial t} + \frac{1}{a\cos\theta}\left(u\frac{\partial v}{\partial \lambda} + v\cos\theta\frac{\partial v}{\partial \theta}\right) + \left(f + \frac{u\tan\theta}{a}\right)u + \frac{g}{a}\frac{\partial h}{\partial\theta} & = & 0 \\
\frac{\partial h}{\partial t} + \frac{1}{a\cos\theta} \left(\frac{\partial (hu)}{\partial \lambda} + \frac{\partial (hv\cos\theta)}{\partial\theta}\right) & = & 0.
\end{eqnarray}
\end{subequations}
Where $f = 2\Omega\sin\theta$, $h$ is the height of the atmosphere, $u$ is the zonal wind component, $v$ is the meridional wind component, $\theta$ and $\lambda$ are the latitudinal and longitudinal directions, $a$ is the radius of the earth, $\Omega$ is the rotational velocity of the earth, and $g$ is the gravitational constant. The space discretization is performed using the unstaggered Turkel-Zwas scheme \cite{Navon_1987_swe, Navon_1991_swe}, with 72 nodes in the longitudinal direction and 36 nodes in the latitudinal direction. 

\begin{figure}[htp]
\centering
\includegraphics[width=4in]{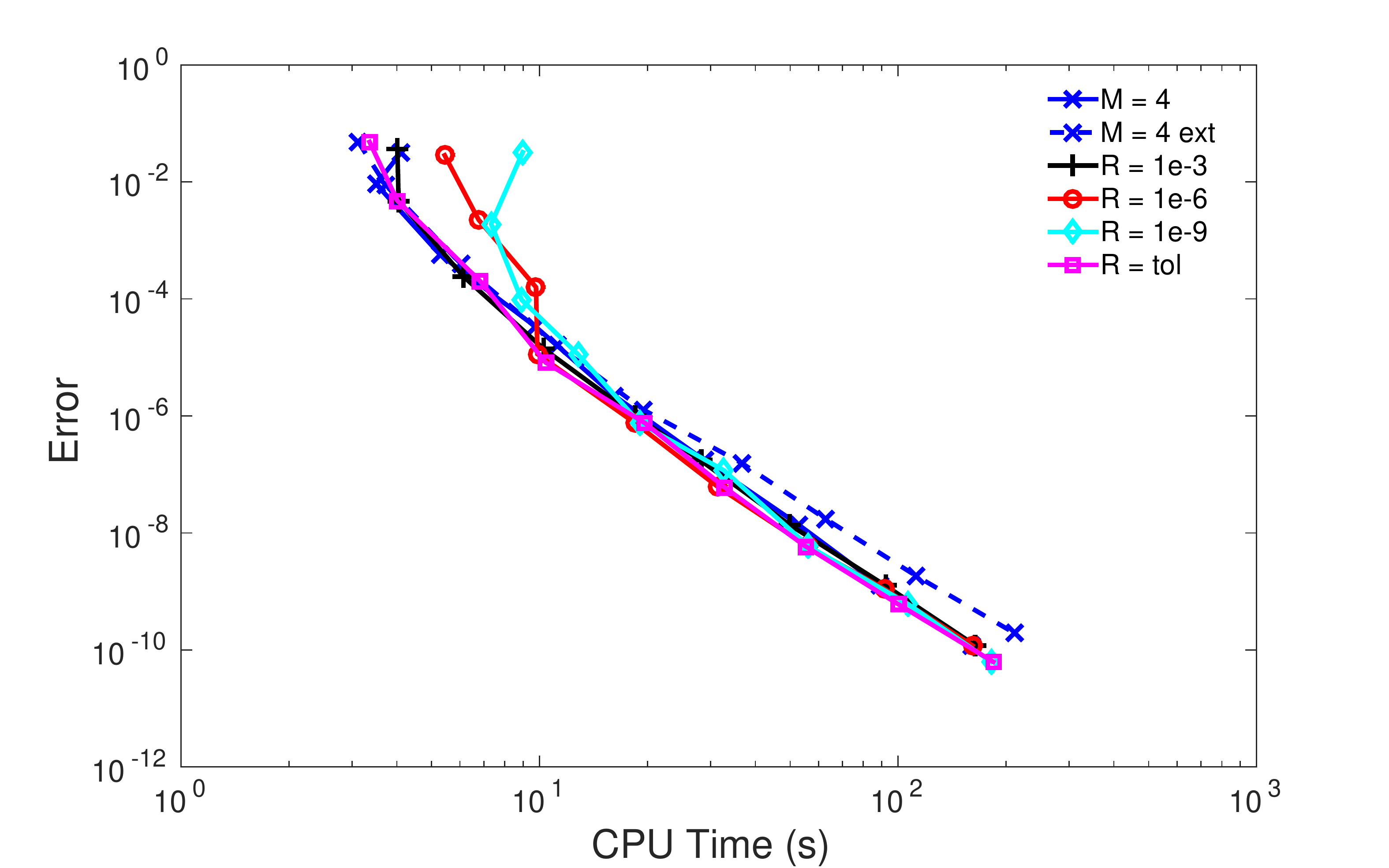}
\caption{Work-precision diagram for a Rosenbrock-Krylov method applied to the shallow water equations \eqref{Radv:eqn:shallowwater}.}
\label{Radv:fig:sweprecision}
\end{figure}

Figure \ref{Radv:fig:sweprecision} shows the relative performance of a fourth order ROK method with a variety of basis selection criteria, applied to the shallow water equations. For non-stiff problems, the primary limiter for stepsize is the error controller, rather than method stability, and making each timestep as cheap as possible will result in the fastest method of a given order. This can be seen in the figure, with the fixed basis $m = 4$ criteria proving the fastest in the shallow water equations. The notable result, however, is that the adaptive basis size methods also select appropriately small basis sizes to be competitive, except when we examine the combination of tight residual tolerance with loose error tolerance. This behavior can be explained if we recall from Corollary \ref{Radv:thm:istageresidual} in Section \ref{Radv:sec:controlresid} above, that the first stage residual scales with $h$, the timestep size, so enforcing a residual $r_{M;1} \leq 10^{-9}$ with a relatively large stepsize $h$ requires more basis vectors. This is the primary motivation behind selecting the residual tolerance to be the same as the method error tolerance in the $R = \text{tol}$ experiment. Finally, note that extension of the basis for non-stiff problems shows no benefit.

\subsection{Stiff case: CBM-IV}
Next we give some results for {\sc ROK} methods applied to a stiff system of ODEs coming from a KPP MATLAB implementation of the CBM-IV model \cite{Gery89photochemicalkinetics}.  This problem is based on the Carbon Bond Mechanism IV (CBM-IV), consisting of 32 chemical species involved in 70 thermal and 11 photolytic reactions \cite{Sandu96benchmarkingstiff}.
%
\begin{figure}[htp]
\centering
\includegraphics[width=4in]{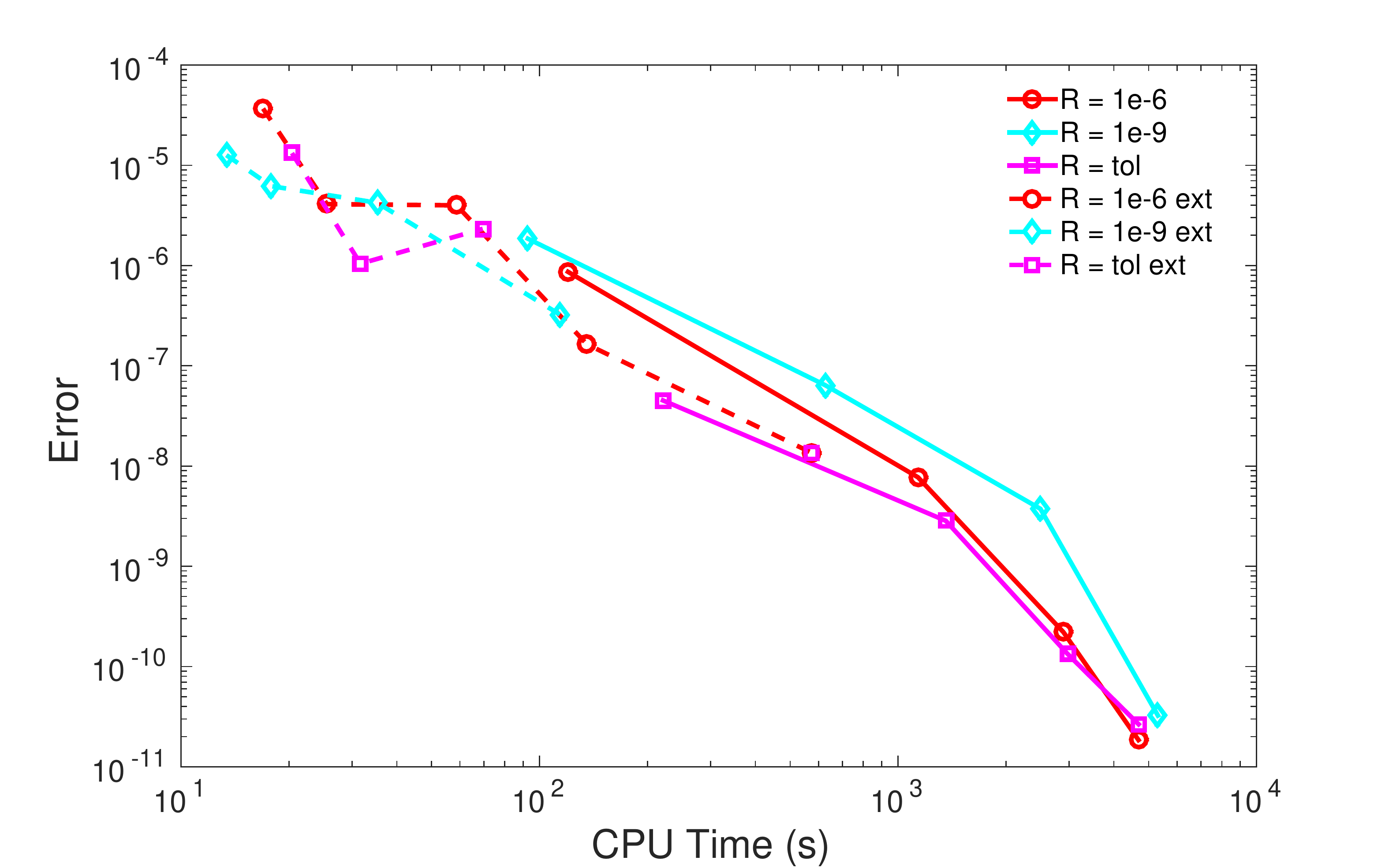}\\
\caption{Precision Diagram for CBM-IV test problem, using adaptive basis size with and without basis extension.}
\label{Radv:fig:cbmiv}
\end{figure}

Figure \ref{Radv:fig:cbmiv} shows the relative performance of a fourth order ROK method with varying tolerances for the first stage residual applied to the CBM-IV test model.  We present here only methods using adaptive basis sizes, since methods using a fixed strategy do not converge for most error tolerance settings, and though the same 9 error tolerance settings were run for each method, only the settings that converged are presented.  Due to the stiffness of the problem, timestep size is clearly limited by the stability of the method, as solutions with accuracy lower than about $10^{-5}$ cannot be obtained. So, this very stiff case demonstrates two important results. First, residual based adaptive basis size methods can solve problems for which fixed basis methods fail to converge entirely. Second, basis extension significantly improves the stability of the method, allowing looser and much faster error tolerance settings to converge.

\subsection{Varying stiffness: the Allen-Cahn problem}

For further performance comparison we consider the two-dimensional Allen-Cahn system, a parabolic partial differential equation 
\begin{eqnarray}
\label{Radv:eqn:allencahn}
&& \frac{\partial}{\partial t} u = \alpha\nabla^2 u + \gamma\left(u - u^3\right), \quad (x,y) \in [0,1] \times [0,1], \quad  t \in [0, 0.2], 
\end{eqnarray}
with $\alpha$ set to $0.1$ or $1.0$ and $\gamma = 1.0$. The problem has homogeneous Neumann boundary conditions and the initial solution 
$u(t=0) = 0.4 + 0.1(x + y) + 0.1\sin(10x)\sin(20y)$.

\begin{figure}[H]
\centering
\subfigure[$N = 64 \times 64$, $\alpha = 0.1$ ]{
\includegraphics[width=4in]{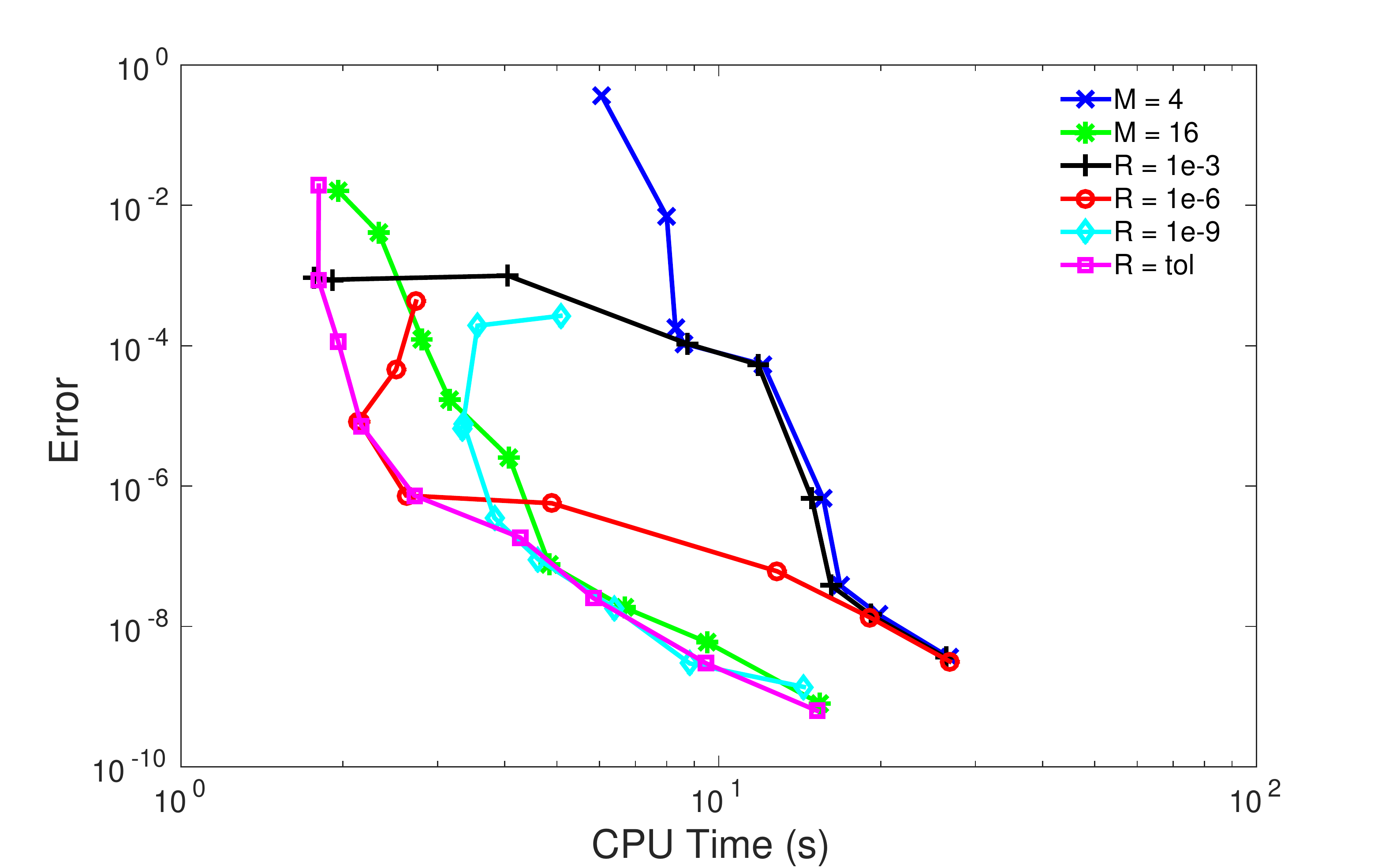}
\label{Radv:fig:acprecisionsmallns}
} \\
\subfigure[$N = 64 \times 64$, $\alpha = 1.0$ ]{
\includegraphics[width=4in]{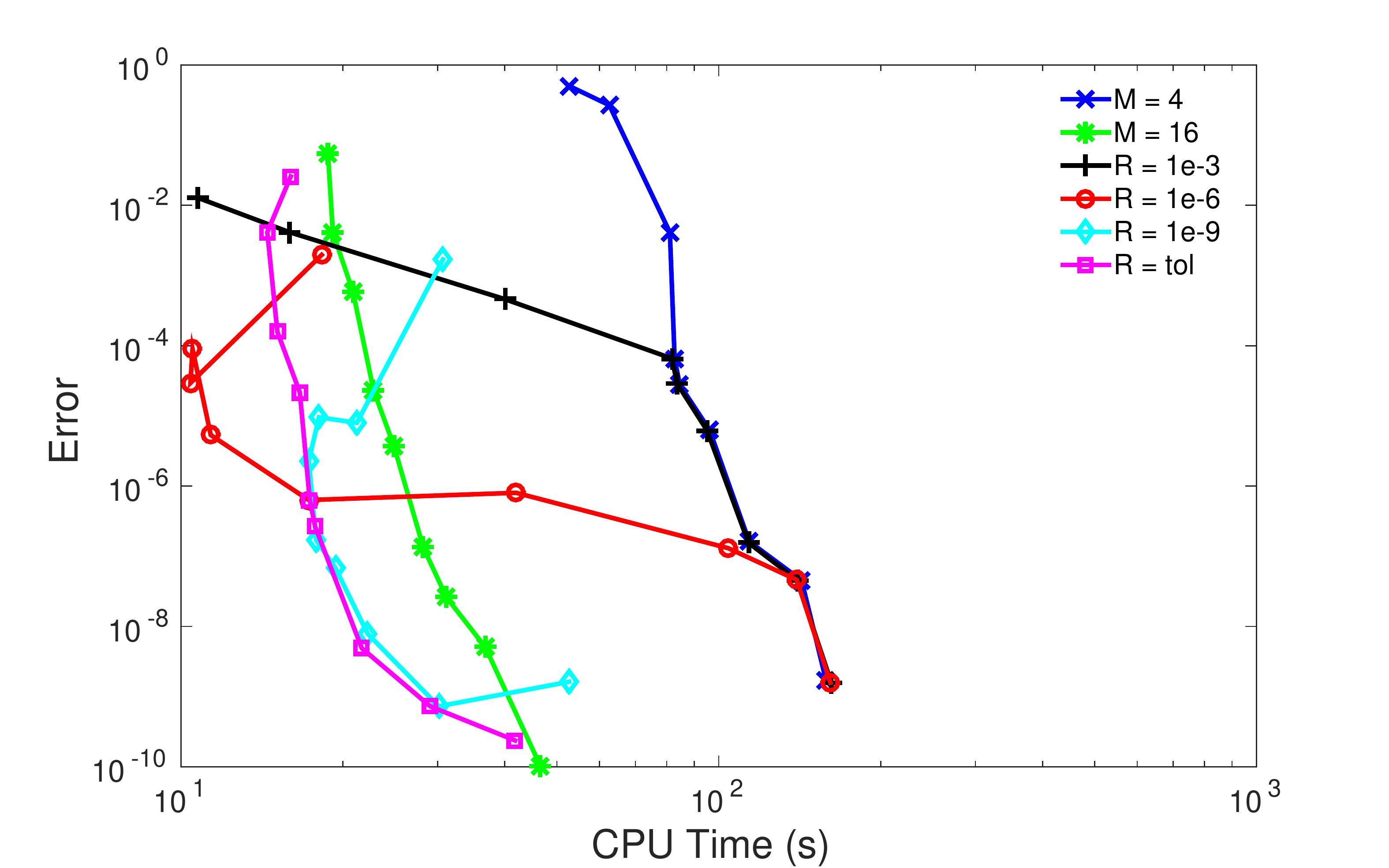}
\label{Radv:fig:acprecisionsmalls}
} \\
\subfigure[$N = 256 \times 256$, $\alpha = 1.0$ ]{
\includegraphics[width=4in]{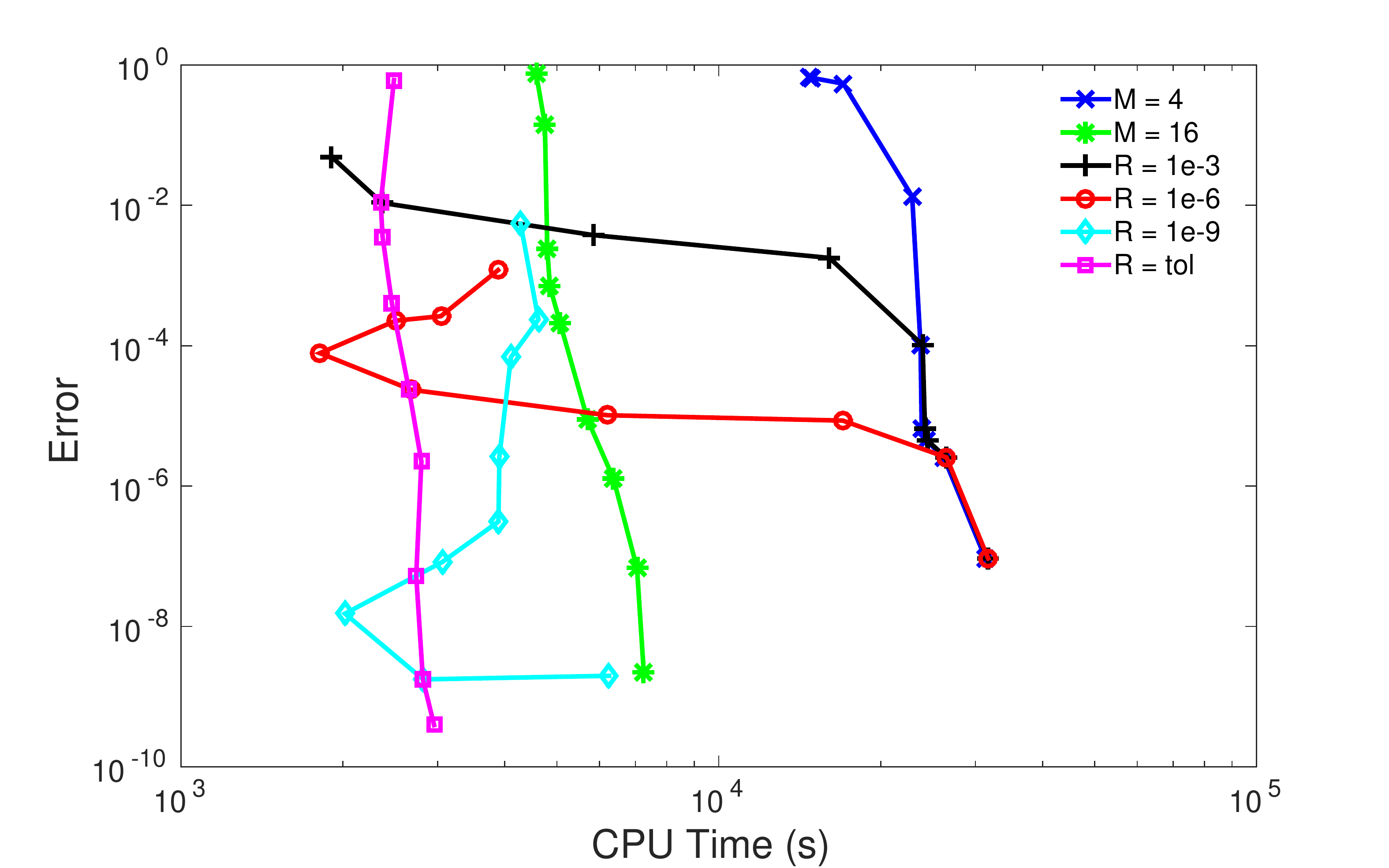}
\label{Radv:fig:acprecisionbigs}
}
\caption{Work-precision diagram for a Rosenbrock-Krylov method applied to the Allen-Cahn test problem \eqref{Radv:eqn:allencahn}.}
\label{Radv:fig:acprecision}
\end{figure}

\begin{figure}[H]
\centering
\subfigure[$N = 64 \times 64$, $\alpha = 0.1$ ]{
\includegraphics[width=4in]{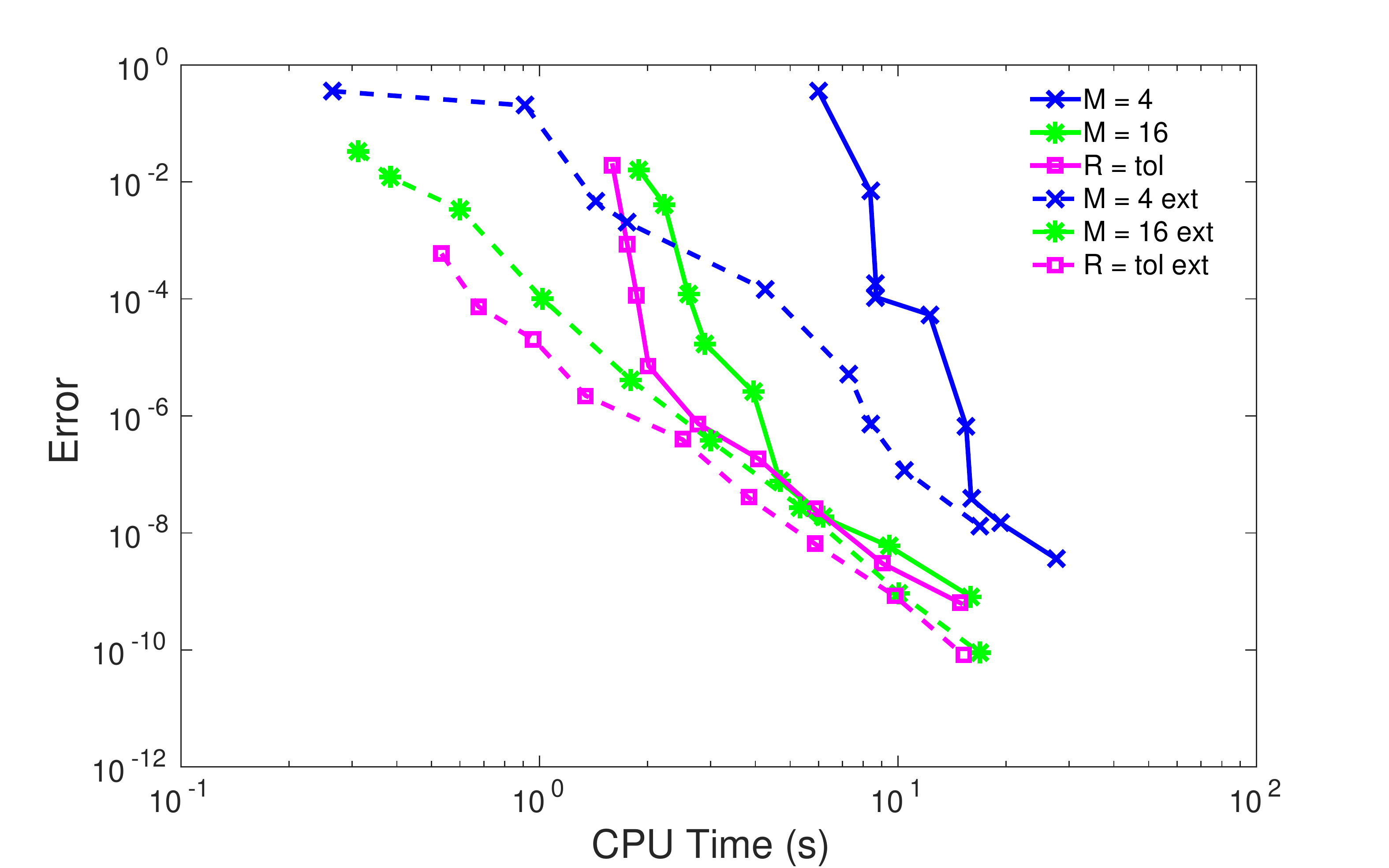}
\label{Radv:fig:acprecisionextsmallns}
} \\
\subfigure[$N = 64 \times 64$, $\alpha = 1.0$ ]{
\includegraphics[width=4in]{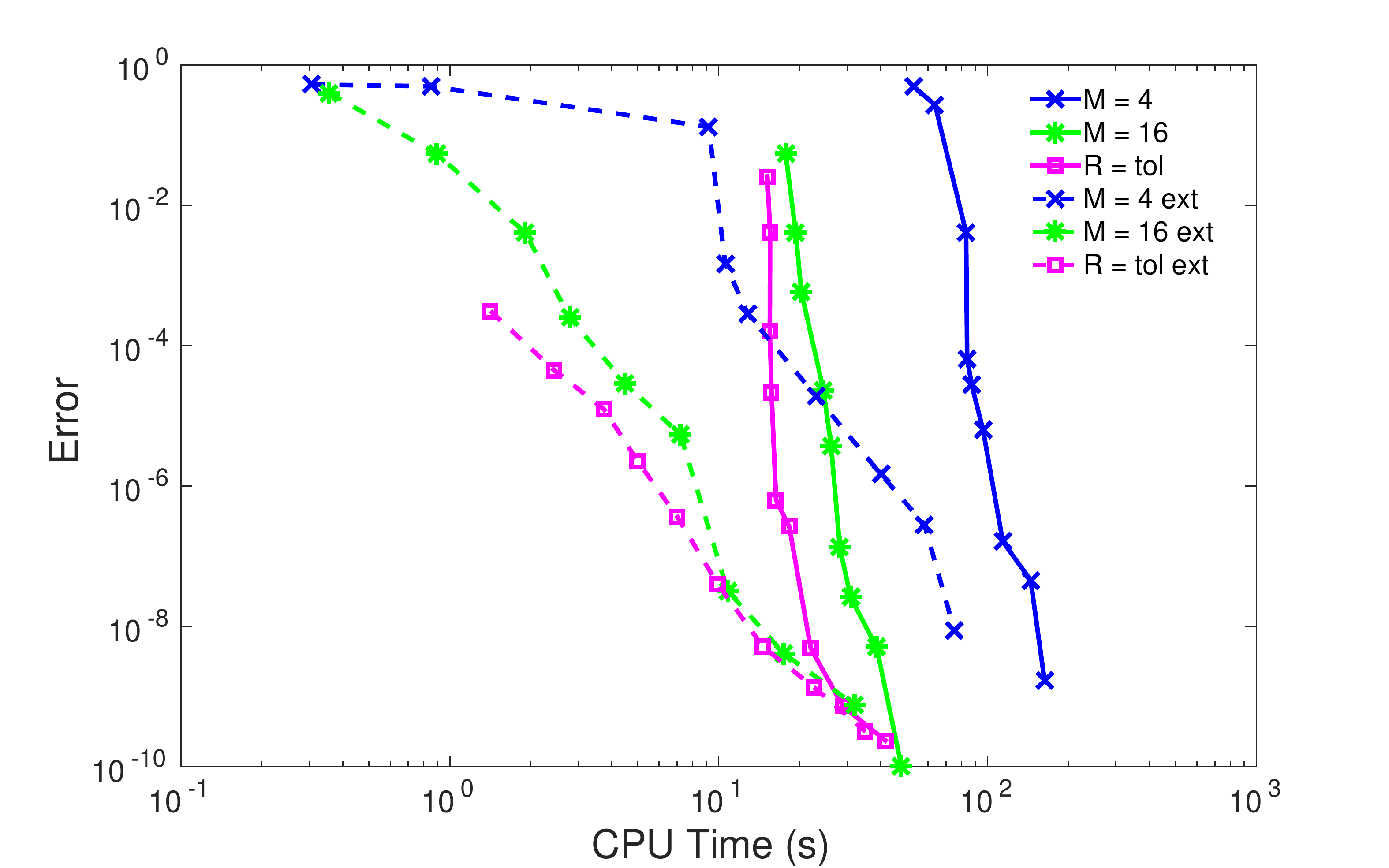}
\label{Radv:fig:acprecisionextsmalls}
} \\
\subfigure[$N = 256 \times 256$, $\alpha = 1.0$ ]{
\includegraphics[width=4in]{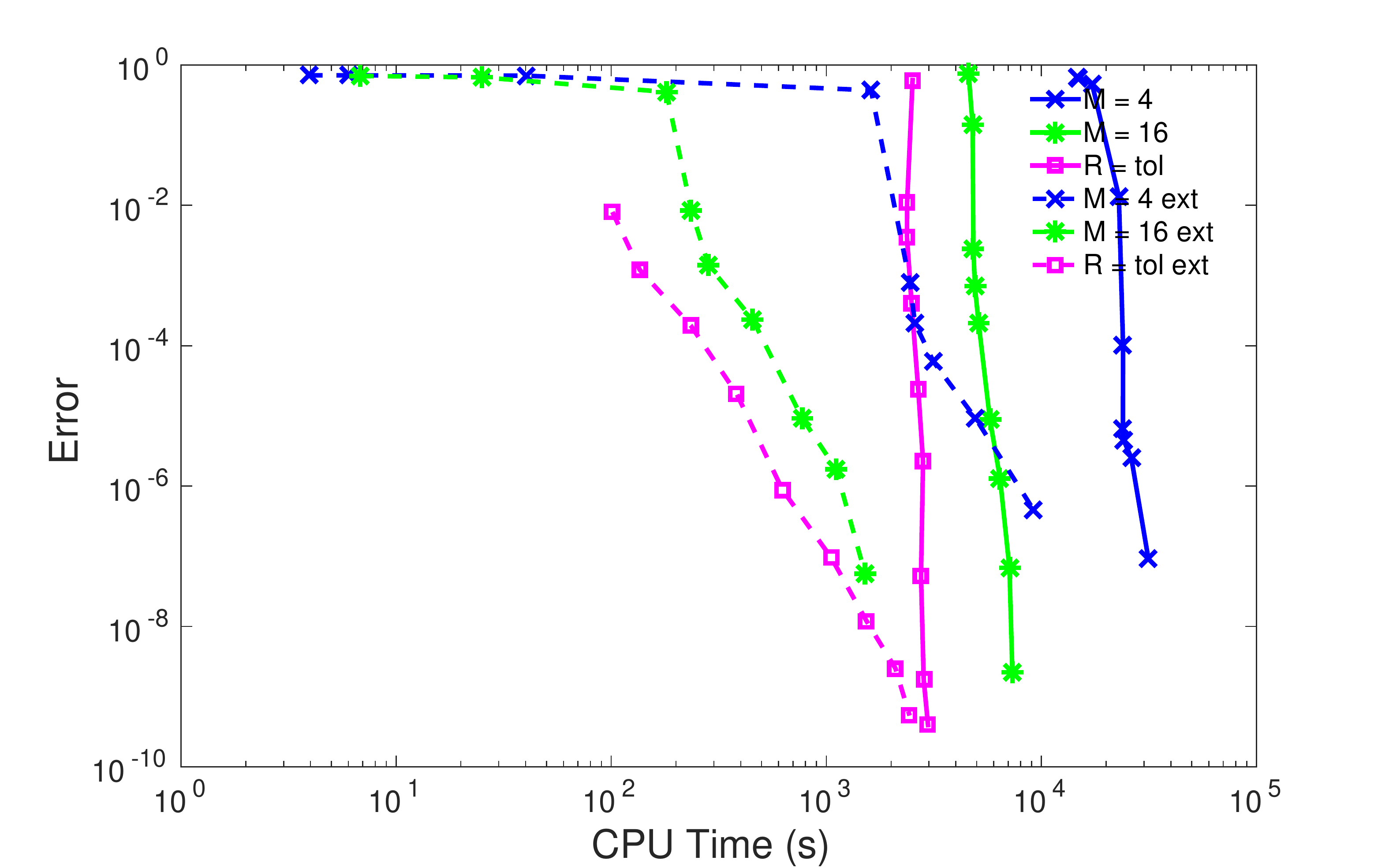}
\label{Radv:fig:acprecisionextbigs}
}
\caption{Work-precision diagram for a Rosenbrock-Krylov method applied to the Allen-Cahn test problem \eqref{Radv:eqn:allencahn}.}
\label{Radv:fig:acprecisionext}
\end{figure}

Figure \ref{Radv:fig:acprecision} compares the relative performance of a fourth order ROK method with varying tolerances for the residual of the first stage applied to the Allen-Cahn problem, alongside some fixed basis size methods for comparison.  Here we have varied the problem from relatively non-stiff in figure \ref{Radv:fig:acprecisionsmallns} to stiff in \ref{Radv:fig:acprecisionbigs} by varying both the spatial resolution and size of the diffusion parameter $\alpha$.  The vertical lines in the figures are evidence of the problem stiffness, resulting from an interaction between the error controller and method stability limiting the size of stable timesteps. Further evidence of stiffness can be seen by examining the fixed basis size results, with $M = 16$ performing 2-3 times better than $M = 4$.  So, in a stability bound problem, we observe the benefit of the residual based adaptive basis size method, with all adaptive methods matching or exceeding the performance of the $M = 4$ baseline as well as the $M = 16$ setting for a number of error tolerances.  Looking closer, we observe the intuition which resulted in the $R = \text{tol}$ heuristic of matching the residual tolerance to the error tolerance: each of the fixed residual tolerance settings perform best when the mismatch between the residual tolerance and error tolerance is small. For example, from Figure \ref{Radv:fig:acprecisionbigs}, the $r_{M;1} \leq 10^{-3}$ setting outperforms all others for error tolerances from $10^{-2}$ to $10^{-3}$, $r_{M;1} \leq 10^{-6}$ proves best from $10^{-4}$ to $10^{-6}$, and $r_{M;1} \leq 10^{-9}$ takes the remainder. So, by setting the residual tolerance and error tolerance equal, we find a setting which produces the most consistent performance over the largest number of error tolerances. 

Figure \ref{Radv:fig:acprecisionext} takes the fixed basis results and $R = \text{tol}$ results from figure \ref{Radv:fig:acprecision} and compares them with and without basis extension.  Clearly, by adding the extension to the basis at each stage, we dramatically increase the stability of the ROK method allowing for larger timesteps, resulting in lines that are less vertical and, for low error tolerances, speedups nearing $10\times$.  Also, combining the basis extension and adaptive basis size gives excellent results, with the $R = \text{tol ext}$ setting providing the best performance for error tolerances tighter than $10^{-4}$. For the stiffest problem in figure \ref{Radv:fig:acprecisionextbigs}, the combination of these enhancements reduces the simulation time from several hours for the fixed basis tests to a few hundred seconds.

\section{Conclusions}
\label{Radv:sec:conclusions}

This work presents two new strategies to adaptively select the underlying Krylov space in Rosenbrock-Krylov schemes.  These approaches are based on a linear stability analysis that considers the impact of approximate solutions of the stage linear systems on the overall stability of the Rosenbrock-Krylov method.  The first strategy is to adaptively extend the dimension of the underlying Krylov basis, in a manner similar to standard approaches for other linearly-implicit methods; however, the motivation here are stability and not accuracy considerations.  The second approach augments the Krylov basis with the stage right hand side vectors; this approach is unique to Rosenbrock-Krylov methods and may provide the foundation for future basis enlargement strategies.  Numerical experiments confirm that both adaptive strategies individually improve the efficiency of the Rosenbrock-Krylov methods, with the basis augmentation strategy being particularly effective.

Future work will focus on developing a theoretical explanation of the dramatic improvements in efficiency provided by augmenting the Krylov space with stage right hand side vectors. Such a theory is likely to provide valuable insight into the stability properties of Rosenbrock-Krylov integrators.  Moreover, a linear stability analysis which leverages the specific form of the Rosenbrock-Krylov scheme may yield new implementation strategies.

\section*{Acknowledgements}

This work has been supported in part by NSF through awards NSF DMS–1419003 and
NSF CCF–1613905, AFOSR through the award
 AFOSR DDDAS 15RT1037,
and by the Computational Science Laboratory at Virginia Tech.

\section*{References}

\bibliographystyle{plain}
\bibliography{Master,ode_krylov,pde_time_implicit,ode_general,ode_Multirate,ode_exponential,ode_rosenbrock,sandu}
\end{document}